\documentclass[11pt,a4paper,oneside]{amsart}
\pdfoutput=1
\usepackage{amssymb}
\usepackage{esint}
\usepackage{mathtools}
\mathtoolsset{showonlyrefs}
\usepackage{hyperref}
\usepackage{enumitem}
\usepackage[english]{babel}
\usepackage[utf8]{inputenc} 
\usepackage[pdftex]{graphicx}

\theoremstyle{plain}
\newtheorem{theorem}{Theorem}[section]
\newtheorem{lemma}[theorem]{Lemma}

\theoremstyle{definition}

\theoremstyle{remark}
\newtheorem{remark}{Remark}[section]

\newcommand{\R}{\mathbb{R}}

\newcommand{\Z}{\mathbb{Z}}
\newcommand{\C}{\mathbb{C}}

\newcommand{\vev}[1]{\left\langle#1\right\rangle}

\renewcommand{\phi}{\varphi}
\newcommand{\abs}[1]{| #1 |}
\newcommand{\norm}[1]{\Vert #1 \Vert}

\DeclareMathOperator{\supp}{supp}

\DeclareMathOperator{\id}{Id}

\title[Fourier magnitude data and Sobolev embeddings]{A note on the Fourier magnitude data and Sobolev embeddings}

\author[J. Railo]{Jesse Railo}
\address{Computational Engineering, School of Engineering Sciences,
Lappeenranta-Lahti University of Technology LUT, Lappeenranta, Finland}
\email{jesse.railo@lut.fi}
\date{April 20, 2024}

\begin{document}

\begin{abstract} We study Sobolev $H^s(\R^n)$, $s \in \R$, stability of the Fourier phase problem to recover $f$ from the knowledge of $\abs{\hat{f}}$ with an additional Bessel potential $H^{t,p}(\R^n)$ a priori estimate when $t \in \R$ and $p \in [1,2]$. These estimates are related to the ones studied recently by Steinerberger in "On the stability of Fourier phase retrieval" \emph{J. Fourier
Anal. Appl.}, 28(2):29, 2022. While our estimates in general are different, they share some comparable special cases and the main improvement given here is that we can remove an additional imaginary term and obtain sharper constants. We also consider these estimates for the quotient distances related to the non-uniqueness of the Fourier phase problem. Our arguments closely follow the Fourier analysis proof of the Sobolev embeddings for Bessel potential spaces with minor modifications.
\end{abstract}
\maketitle

\subsection*{Acknowledgements} The author was supported by the Research Council of Finland through the Flagship of Advanced Mathematics for Sensing, Imaging and Modelling (decision number 359183). The author acknowledges support by the Swiss National Science Foundation (SNF Grant 184698) and Rima Alaifari at ETH Zürich in 2020--2021. The author is grateful to Stefan Steinerberger for comments related to \cite{Stei20-stability} and this note.

\section{Introduction}

The Fourier phase problem studies the recovery of information about a function $f$ from the knowledge of the magnitude of its Fourier transform $\abs{\hat{f}}$ where the Fourier transform is defined with the convention
\begin{equation}
    \hat{f}(\xi) := \mathcal{F}f(\xi) := (2\pi)^{-n/2}\int_{\R^n} e^{-ix \cdot \xi}f(x)dx.
\end{equation} This problem has non-unique solutions for several reasons. For example, $f$ and $cf$ have the same Fourier magnitude data for any constant $c \in \C$ with $\abs{c}=1$. See the discussion in Section \ref{sec:quotientDist}, Appendix \ref{sec:genQuotientL2}, and the survey \cite{GKR20-PhaseRetSurvey} for other obstructions of uniqueness.

The first systematic mathematical studies of Fourier phase retrieval are due to Akutowicz \cite{Aku56-PhaseretI,Aku57-PhaseretII}, Walther \cite{Wal63-PhaseRet}, and Hofstetter \cite{Hof64-PhaseRet}, where the uniqueness questions were considered in one dimension for compactly supported functions. Possible imaging applications include X-ray crystallography, transmission electron microscopy, and coherent diffractive imaging \cite{BBE17-ApplPhaseRetSurvey,Fie82-Appl,GKR20-PhaseRetSurvey}. Other works on the uniqueness properties of the Fourier phase problem include, for example,  \cite{CF83-FourierPhaseRetUniq,KST95-FourierPhaseRet}.

In this paper, we consider inequalities related to the stability of Fourier phase retrieval. Our main result gives mixed estimates for general $H^s(\R^n)$, $s \in \R$, quotient distances related to the non-uniqueness of the Fourier phase problem using the Fourier magnitude difference data $\abs{\abs{\hat{f}}-\abs{\hat{g}}}$ and certain a priori estimates given in terms of the Bessel potential norms $H^{t,p}(\R^n)$, $p \in [1,2]$, $t \in \R$. This result is stated in Theorem \ref{thm:HsBesselGroupthm}. This improves some of the recent results in \cite{Stei20-stability} as explained in the discussion of Section \ref{sec:motivation}. One may remove the a priori estimates if $\supp(\hat{f}) \neq \supp(\hat{g})$ but such estimates seem to hold for almost trivial reasons and without good control of constants, as Appendix \ref{sec:appendix1} demonstrates.

The question of stability for Fourier phase retrieval has been studied recently in \cite[Theorem 4.13]{GKR20-PhaseRetSurvey} with additional measurements and in \cite{Stei20-stability} with additional a priori information. The stability of other or general phase retrieval problems has been a very active topic of study in the last few years. See, for example, \cite{ADGY19-StabPhaseRetInfDim,AG17-PhaseRetBanachFrames,AG21-GaborPhaseRetIllPosed,AW21-StabDiscGabPhaseRet,BW15-GeneralPhaseRetStab,BCMN14-SavingPhase,CCD16-PhaseRetInfHilbSpace,CDL20-StabPhaseLocStabCondConnected,GR19-StabGaborPhaseRet,GR19-StabGaborPhaseRetMultivar,IMP19-LowerLipBoundsPhaseRet} and the references therein.

We introduce the main result in Section \ref{sec:mainres}. We consider preliminaries in Section \ref{sec:preliminaries}. We give the proof of our main theorem in Section \ref{sec:mainthm}. We consider other aspects of the stability of the Fourier phase problem in Appendices \ref{sec:appendix1} and \ref{sec:genQuotientL2} in some simpler special cases. 

\subsection{Main result} \label{sec:mainres}
We briefly introduce the notation used first and then state our main theorem. More details are given in Section \ref{sec:preliminaries}. We denote $\vev{\xi} = \sqrt{1+\abs{\xi}^2}$ for any $\xi \in \R^n$ and define \emph{Bessel potential spaces} for any $t \in \R$ and $p \in [1,\infty]$ as
\begin{equation}
    H^{t,p}(\R^n) := \{\, f \in \mathcal{S}'(\R^n)\,;\, \mathcal{F}^{-1}(\vev{\cdot}^t \hat{f}) \in L^p(\R^n)\,\}
\end{equation}
equipped with the norm $\norm{f}_{H^{t,p}(\R^n)} := \norm{\mathcal{F}^{-1}(\vev{\cdot}^t \hat{f})}_p$. For the space $H^s(\R^n)$, we directly use the Fourier side norm $\norm{\vev{\cdot}^s \hat{f}}_2$, which is equal to the norm of $H^{s,2}(\R^n)$ up to a positive constant. Furthermore, for  $f \in \mathcal{S}'(\R^n)$, we use the convention that $\norm{f}_{H^{t,p}(\R^n)}=\infty$ if $f \notin H^{t,p}(\R^n)$.

We denote by $\mathcal{F}^s(\R^n)$ the group of bijective and additive operators (not necessarily linear) $P: H^s(\R^n) \to H^s(\R^n)$ with the property \begin{equation} \abs{\hat{f}(\xi)} = \abs{\widehat{Pf}(\xi)} \quad \forall \xi \in \R^n, \quad f \in H^s(\R^n).\end{equation} The related group operation is the usual composition of mappings. The group is $\mathcal{F}^s(\R^n)$ is also known as the group of trivial solutions of the Fourier phase problem. If $G$ is a subgroup of $\mathcal{F}^s(\R^n)$, then the distance of the related quotient is defined as 
\begin{equation} d ([f], [g]) = \inf_{P \in G} \norm{f-Pg}_{H^s(\R^n)} 
\end{equation}
for the equivalence classes of $f,g \in H^s(\R^n)$: $[f]=[g]$ if and only if $f=Pg$ for some $P \in G$. See Section \ref{sec:quotientDist} for a brief discussion of the non-uniqueness of the Fourier phase problem.

Our main theorem is an estimate which combines the fractional Sobolev inequalities with the Fourier magnitude data. We write $$A_{f \cap g} = \supp(\hat{f})\cap\supp(\hat{g})$$ and use the convention $\frac{2}{2-2} = \infty$ throughout the article. We also write $$c_{n,p} := \norm{\mathcal{F}}_{L^p \to L^{p/(p-1)}}^2$$ for the squared operator norm of the Fourier transform. Note that $c_{n,p}\leq1$.

\begin{theorem}\label{thm:HsBesselGroupthm} Let $n \in \Z_+$, $t, s\in \R$ and $p \in [1,2]$. Let $G \subset \mathcal{F}^s(\R^n)$ be a subgroup. Suppose that $f,g \in H^s(\R^n)$. Then
\begin{equation}
    d([f],[g])^2 \leq \norm{\vev{\xi}^s(|\hat{f}|-\abs{\hat{g}})}_2^2 + R(f,g)
\end{equation}
where
\begin{equation}
R(f,g) := c_{n,p}\norm{\chi_{A_{f \cap g}}\vev{\xi}^{2s-2t}}_{\frac{p}{2-p}}\inf_{P,Q \in G}\norm{Pf-Qg}_{H^{t,p}(\R^n)}^2
\end{equation}
depends on $n,p,t,s$ and $G$.

The constant $\norm{\chi_{A_{f \cap g}}\vev{\xi}^{2s-2t}}_{\frac{p}{2-p}}$ is finite if one of the following sufficient conditions hold:
\begin{enumerate}[label=(\roman*)]
\item $A_{f \cap g}$ is compact.
\item $A_{f \cap g}$ has finite measure and $s \leq t$. 
\item $s < t-a$ where $a = n(\frac{1}{p}-\frac{1}{2})$.
\end{enumerate}
\end{theorem}
\begin{remark}
If $G=\{\id\}$, then $d([f],[g]) = \norm{f-g}_{H^s(\R^n)}$ on the left-hand side and $P=Q=\id$ on the right-hand side. See the formulation given in Lemma \ref{lem:HsBesselthm}.
\end{remark}

Note that the sharp constants for the Fourier transform are obtained in an article by Beckner in 1975 \cite[Theorem 1]{Beckner-annals} but the normalization of the Fourier transform is not the same as used in this article. Therefore, we do not present the explicit constant $c_{n,p}$ here, but simply indicate that it could be obtained. Practically, one could simply replace $c_{n,p}$ by $1$ everywhere and the estimates presented in this article would still hold. We discuss some special cases in Remarks \ref{rmk:finitecase}, \ref{rmk:simplecase} and \ref{rmk:simplecase2}.

\begin{remark}\label{rmk:finitecase} If $\supp(\hat{f})$ has finite measure and $s \leq t$, then the constant can be chosen to be such that it does not depend on particular $g$ since $\abs{A_{f \cap g}} \leq \abs{\supp(\hat{f})}$.
\end{remark}

\begin{remark}\label{rmk:simplecase}
If we estimate when $s,t=0$ and $G = \{\id\}$, then we can recover the estimate
\begin{equation}
    \norm{f-g}_2^2 \leq \norm{\abs{\hat{f}}-\abs{\hat{g}}}_2^2 + c\abs{\supp(\hat{f}) \cap \supp(\hat{g})}^{\frac{2-p}{p}}\norm{f-g}_p^2.
\end{equation} If additionally $p=1$, then
\begin{equation}
    \norm{f-g}_2^2 \leq \norm{\abs{\hat{f}}-\abs{\hat{g}}}_2^2 + (2\pi)^{-n}\abs{\supp(\hat{f}) \cap \supp(\hat{g})}\norm{f-g}_1^2.
\end{equation}
\end{remark}

\begin{remark}\label{rmk:simplecase2}
If we estimate when $G = \{\,w \,;\, w \in S^1 \subset \C\}$, then we can recover the estimate
\begin{equation}
\begin{split}
    \inf_{w \in S^1} \norm{f-wg}_{H^s(\R^n)}^2
    &\leq \norm{\vev{\xi}^s(|\hat{f}|-\abs{\hat{g}})}_2^2 \\
    &\quad+ c_{n,p}\norm{\chi_{A_{f \cap g}}\vev{\xi}^{2s-2t}}_{\frac{p}{2-p}}\inf_{w \in S^1}\norm{f-wg}_{H^{t,p}(\R^n)}^2
    \end{split}
\end{equation}
for $p \in [1,2]$ and $t \in \R$. We additionally used here that $w: f \mapsto wf$ is an isometry also on $H^{t,p}(\R^n)$.
\end{remark}

The proof of Theorem \ref{thm:HsBesselGroupthm} is given in Section \ref{sec:mainthm} but we outline it next. Notice that if $f \in L^1(\R^n)\cap L^2(\R^n)$ is such that the support of $\hat{f}$ has measure $L$, then
\begin{equation}\label{eq:basicestimate}
    \norm{f}_{L^2(\R^n)}^2 = \norm{\hat{f}}_{L^2(\R^n)}^2 \leq L\norm{\hat{f}}_{L^\infty(\R^n)}^2 \leq (2\pi)^{-n}L\norm{f}_{L^1(\R^n)}^2.
\end{equation}
We take a few additional steps to deal with the difference of magnitudes of the Fourier transforms of two functions. The fractional Sobolev norm estimates are based on typical manipulations of the related Fourier multipliers. The inequality for the quotient distances follows from the case $G = \{\id\}$, and hence this basic case can be considered to be more important.

We make a few remarks about the estimate of Theorem \ref{thm:HsBesselGroupthm} as it looks similar to the fractional Sobolev inequalities for Bessel potential spaces \cite{TRI-interpolation-function-spaces} for $q \geq p$:
\begin{itemize}
\item If $t-s \geq n(\frac{1}{p}-\frac{1}{q})$, $p,q \in (1,\infty)$, $s,t \in \R$, then
\begin{equation}\label{eq:secondSob}
    \norm{f}_{H^{s,q}(\R^n)} \leq C\norm{f}_{H^{t,p}(\R^n)} \quad \forall f \in H^{t,p}(\R^n).
\end{equation}
    \item If $t-s > n(\frac{1}{p}-\frac{1}{2})$, $p \in [1,2]$, $s, t \in \R$, then
\begin{equation}\label{eq:firstSob}
    \norm{f}_{H^s(\R^n)} \leq C\norm{f}_{H^{t,p}(\R^n)} \quad \forall f \in H^{t,p}(\R^n).
\end{equation}
\end{itemize}
Recall that the embedding \eqref{eq:firstSob}, when $1 \leq p \leq 2$, can be proved similarly to the proof of Lemma \ref{lem:HsBesselthm} using first Parseval's identity and then the Hausdorff--Young and Hölder's inequalities. The usual proof of cases $p > 2$ requires more advanced techniques that utilize the classical theory of interpolation spaces and Fourier multipliers.

\begin{remark} We note that the main result with $s < t-a$ and $G = \{\id\}$ is almost the fractional Sobolev inequality but now there is an improvement by the factor generated by $\chi_{A_{f \cap g}}$ in the coefficient due to the use of the Fourier magnitude data.
\end{remark}

\begin{remark} The case where at least one of the functions $\hat{f}$ or $\hat{g}$ has a support of infinite measure but $\supp(\hat{f}) \cap \supp(\hat{g})$ has finite measure or is compact gives an estimate also when the Sobolev inequalities do not directly apply to the function $f-g$. In this case, the additional Fourier magnitude data permits the estimate of Theorem \ref{thm:HsBesselGroupthm}.
\end{remark}

\subsection{Comparison to the work \cite{Stei20-stability}}\label{sec:motivation} We briefly recall some parts of the work \cite{Stei20-stability} to motivate this note. Let $f \in L^1(\R^n) \cap L^2(\R^n)$ be a function with real-valued Fourier transform. Let \begin{equation}L = \abs{\{\,\xi \in \R^n \,;\, \hat{f}(\xi)\neq0\,\}} < \infty.\end{equation}
Then for any $g \in L^1(\R^n) \cap L^2(\R^n)$ one has the estimate (\cite[Corollary 1]{Stei20-stability})
\begin{equation}\begin{split}
    &\norm{f-g}_{L^2(\R^n)} \\&\quad\leq 2 \norm{|\hat{f}|-|\hat{g}|}_{L^2(\R^n)} + 30 \sqrt{L} \norm{f-g}_{L^1(\R^n)} +2\norm{\Im \hat{g}}_{L^2(\R^n)}.\label{eq:stei20cor}
    \end{split}
\end{equation}
See Remarks \ref{rmk:finitecase} and \ref{rmk:simplecase} for the formulations of Theorem \ref{thm:HsBesselGroupthm} in the most basic cases to make a comparison with \eqref{eq:stei20cor}.

It is said in \cite{Stei20-stability} that the imaginary term in \eqref{eq:stei20cor} accounts for the translation symmetry in the Fourier phase problem. One motivation for writing this note is that since $L^p$ differences are not invariant under translations of only one of the involved functions, this additional imaginary term coming from the translation symmetry of the Fourier phase problem might not be really required in the estimates of this type. Indeed, this is the case as Theorem \ref{thm:HsBesselGroupthm} shows.

In \cite{Stei20-stability} it is argued that the constant $\sqrt{L}$ is optimal in \eqref{eq:stei20cor} up to the independent constant. We give a completely different proof which improves these independent constants and also $L$ can be replaced by a sharper term when we let the coefficient also depend on the other function $g$ (giving the term $L$ as an upper bound independent of $g$). Furthermore, in our approach it does not make any difference if the involved functions or their Fourier transforms are real or complex valued. We indicate a possible direction for future research related to the estimates of this kind: \emph{Could methods in this note be combined with the arguments in \cite{Stei20-stability} to obtain more general inequalities than presented in \cite{Stei20-stability} and here?}

\section{Preliminaries}
\label{sec:preliminaries}

\subsection{Bessel potentials}
We denote $\vev{\xi} = \sqrt{1+\abs{\xi}^2}$ for any $\xi \in \R^n$ and let $\mathcal{S}'(\R^n)$ be the space of \emph{Schwartz distributions} (i.e. tempered distributions). We denote by $\mathcal{F}$ and $\mathcal{F}^{-1}$ the \emph{Fourier transform} and the \emph{inverse Fourier transform} on $\mathcal{S}'(\R^n)$. We may also write $\mathcal{F}f = \hat{f}$ when $f \in \mathcal{S}'(\R^n)$.  

We can define the $L^2$ Sobolev spaces for any $s \in \R$ using the Fourier transform of distributions as \[H^s(\R^n) = \{\,u \in \mathcal{S}'(\R^n) \,;\, \mathcal{F}^{-1}(\vev{\cdot}^{s}\hat{u}) \in L^2(\R^n)\,\}\]
where we equip the space with the norm $\norm{u}_{H^s(\R^n)} := \norm{\vev{\cdot}^{s}\hat{u}}_2$ as the Fourier transform is an isometry on $L^2(\R^n)$.

We define \emph{Bessel potential spaces} for any $t \in \R$ and $p \in [1,\infty]$ as
\begin{equation}
    H^{t,p}(\R^n) := \{\, f \in \mathcal{S}'(\R^n)\,;\, \mathcal{F}^{-1}(\vev{\cdot}^t \hat{f}) \in L^p(\R^n)\,\}
\end{equation}
equipped with the norm $\norm{f}_{H^{t,p}(\R^n)} := \norm{\mathcal{F}^{-1}(\vev{\cdot}^t \hat{f})}_p$. These are Banach spaces for any $p \in [1,\infty]$ and $\mathcal{S}(\R^n)$ is dense in $H^{t,p}(\R^n)$ if $p \in [1,\infty)$. If $f \in \mathcal{S}'(\R^n)$ and $f \notin H^{t,p}(\R^n)$, then we define $\norm{f}_{H^{t,p}(\R^n)}=\infty$.


One may recall some basic facts of Bessel potentials and related topics, for example, from \cite{Fri98-intro-dist-theory-book,HO-analysis-of-pdos,TRI-interpolation-function-spaces,WO-pseudodifferential-operatros}.

\subsection{Quotient distances with respect to isometry groups}\label{sec:quotientDist}
It is well-known that there exist many operators $P$ such that \begin{equation}\label{eq:phaseProp}\abs{\hat{f}} = \abs{\widehat{Pf}}.\end{equation} Any such operator, not equal to the identity, breaks down the possibility for uniqueness in the Fourier phase problem. One usually says that the following operations are the \emph{trivial ambiguities} for the Fourier phase problem:
\begin{enumerate}[label=(\roman*)]
    \item the multiplication operator $f \mapsto cf(x)$ with $c \in S^1 \subset \C$.
    \item the translation operator $f \mapsto f(x-\tau)$ with $\tau \in \R^n$.
    \item the conjugate reflection operator $f \mapsto \overline{f(-x)}$.
\end{enumerate}
It is easy to check that the above operators have the property \eqref{eq:phaseProp}. We also remark that they all are isometries $L^p(\R^n) \to L^p(\R^n)$ for any $p \in [1,\infty]$ and $H^s(\R^n) \to H^s(\R^n)$ for any $s \in \R$. See also the recent survey \cite{GKR20-PhaseRetSurvey} which discusses in detail the uniqueness and stability of the Fourier phase problem.

Here we define a general setting for including such obstructions of uniqueness in the estimates. This of course leads to defining quotient distances. The right quotient distance may depend on the context. Therefore, we build our theory here such that it is possible to adapt into different cases. 

Let us denote by $\mathcal{B}^s(\R^n)$ the set of bijective operators $H^s(\R^n) \to H^s(\R^n)$ that are additive but not necessarily linear. Let us further denote by $\mathcal{F}^s(\R^n) \subset \mathcal{B}^s(\R^n)$ the subgroup of operators $P \in \mathcal{B}^s(\R^n)$ with the property \begin{equation}\label{eq:FourMagInv} \abs{\hat{f}(\xi)} = \abs{\widehat{Pf}(\xi)} \quad \forall \xi \in \R^n.\end{equation} It follows from Parseval's identity that any bijective and additive $P: H^s(\R^n) \to H^s(\R^n)$ with the property \eqref{eq:FourMagInv} must be an isometry, and it is easy to check that $\mathcal{F}^s(\R^n)$ is a group. Note that these operators do not have to be necessarily linear. For example, the conjugate reflection operator is a conjugate-linear isometry with the property \eqref{eq:FourMagInv}.

We will make use of subgroups $G \subset \mathcal{F}^s (\R^n)$, i.e. subsets with the following property
\begin{enumerate}
    \item If $P, Q \in G$, then $P^{-1}Q \in G$.
\end{enumerate}
We define the quotient space $X^s(G) := H^s(\R^n)/G$ using the equivalence $[f] = [g]$ if and only if $f=Pg$ for some $P \in G$. The space $X^s(G)$ is equipped with the quotient \emph{pseudometric}
\begin{equation}
    d([f],[g]) := \inf_{P,Q \in G} \norm{Pf-Qg}_{H^s(\R^n)}.
\end{equation}
Since $G$ is a group, $P^{-1} \in G$ is an isometry when $P \in G$, and we have for any $P,Q \in G$ that \begin{equation}
    \norm{Pf-Qg}_{H^s(\R^n)} = \norm{P^{-1}(Pf-Qg)}_{H^s(\R^n)} = \norm{f-P^{-1}Qg}_{H^s(\R^n)}
\end{equation}
where $P^{-1}Q \in G$. The second equality was the point where we used the additivity assumption. One can now conclude the identity \begin{equation}d([f],[g]) = \inf_{P \in G} \norm{f-Pg}_{H^s(\R^n)}.\end{equation} In the presence of certain compactness assumptions on $G$, the pseudometric $d$ would be a metric and $d([f],[g])=0$ if and only if $f = Pg$ for some $P \in G$. For instance, when $G = S^1$ (acting as a scalar multiplication) or finite.

See Appendix \ref{sec:genQuotientL2} for a simple example of how the stability of the Fourier phase problem with respect to quotient metrics may become almost trivial if the group $G$ is taken to be very large.

\section{Proof of the main theorem}
\label{sec:mainthm}

\subsection{Preparation}
Let $f,g \in \mathcal{S}'(\R^n)$. Let us define the following subsets of $\R^n$:
\begin{equation}
    \begin{split}
        A_{f \cap g} &:= \supp(\hat{f})\cap \supp(\hat{g}),\quad
        A_{f \cup g} := \supp(\hat{f})\cup \supp(\hat{g}) \\
        A_{f\setminus g} &:= \supp(\hat{f})\setminus \supp(\hat{g}),\quad\quad\,
        A_f := \supp(\hat{f}).
    \end{split}
\end{equation}
Clearly, \begin{equation}\label{eq:disjointness}
    \R^n = A_{f\cap g} \cup A_{f\setminus g} \cup A_{g \setminus f} \cup A_{f \cup g}^c
\end{equation}
as a disjoint union.

If $A \subset \R^n$, we denote the charateristic function of the set $A$ by $\chi_A$. We define the associated Fourier multiplier $M_A$ by the formula \begin{equation}M_A := \mathcal{F}^{-1}\circ \chi_A \mathcal{F}\end{equation} whenever the definition makes sense. In particularly, if $A$ is a measurable set, then $M_A: H^s(\R^n) \to H^s(\R^n)$ is a bounded linear map with the operator norm at most $1$ for any $s \in \R$.

For the sets $A_{f \cap g}$, $A_{f \cup g}$, $A_{f\setminus g}$ and $A_f$, we denote the related Fourier multipliers as $M_{f \cap g}$, $M_{f \cup g}$, $M_{f\setminus g}$ and $M_f$ respectively.

\subsection{Proof of Theorem \ref{thm:HsBesselGroupthm}}

We first state a simple estimate to be used later. The proof is rather straightforward and thus omitted.

\begin{lemma}\label{lem:H^slemma} Let $n \in \Z_+$ and $s \in \R$. Suppose that $f,g \in H^s(\R^n)$. Then
\begin{equation}\norm{f-g}_{H^s(\R^n)}^2 \leq  \norm{\vev{\xi}^{s}(\abs{\hat{f}}-\abs{\hat{g}})}_2^2+\norm{M_{f\cap g}(f-g)}_{H^s(\R^n)}^2.
\end{equation}
\end{lemma}

We can now prove our main lemma for Theorem \ref{thm:HsBesselGroupthm}. We also remark that Lemma \ref{lem:HsBesselthm} corresponds to the case $G = \{\id\}$ in the statement of Theorem \ref{thm:HsBesselGroupthm}.

\begin{lemma}\label{lem:HsBesselthm} Let $n \in \Z_+$, $t, s\in \R$ and $p \in [1,2]$. Suppose that $f,g \in H^s(\R^n)$ and $f-g \in H^{t,p}(\R^n)$. Then
\begin{equation}
    \norm{f-g}_{H^s(\R^n)}^2 \leq \norm{\vev{\xi}^s(|\hat{f}|-\abs{\hat{g}})}_2^2 + c_{n,p}\norm{\chi_{A_{f \cap g}}\vev{\xi}^{2s-2t}}_{\frac{p}{2-p}}\norm{f-g}_{H^{t,p}(\R^n)}^2
\end{equation}
where $c_{n,p} := \norm{\mathcal{F}}_{L^p \to L^{p/(p-1)}}^2 \leq 1$.

The constant $\norm{\chi_{A_{f \cap g}}\vev{\xi}^{2s-2t}}_{\frac{p}{2-p}}$ is finite if one of the following sufficient conditions hold:
\begin{enumerate}[label=(\roman*)]
\item $A_{f \cap g}$ is compact.
\item $A_{f \cap g}$ has finite measure and $s \leq t$. 
\item $s < t-a$ where $a = n(\frac{1}{p}-\frac{1}{2})$.
\end{enumerate}
\end{lemma}

\begin{proof}
We have by Lemma \ref{lem:H^slemma} that
\begin{equation}\norm{f-g}_{H^s(\R^n)}^2 \leq  \norm{\vev{\xi}^{s}(\abs{\hat{f}}-\abs{\hat{g}})}_2^2+\norm{M_{f\cap g}(f-g)}_{H^s(\R^n)}^2.
\end{equation} The rest of the argument is to control the term $\norm{M_{f\cap g}(f-g)}_{H^s(\R^n)}^2$. 

When $p \in [1,2]$, it follows that \begin{equation}\label{eq:continuity-of-Fourier}\norm{\vev{\cdot}^t \hat{f}}_{p'} \leq \norm{\mathcal{F}}_{L^p \to L^{p'}} \norm{f}_{H^{t,p}(\R^n)}\end{equation} where $p'$ is the Hölder conjugate of $p$. We have by Hölder's inequality that
\begin{equation}\label{eq:hölder-inequality}\int_{A_{f \cap g}}\vev{\xi}^{2s} \abs{\hat{f}-\hat{g}}^2d\xi \leq \norm{\chi_{A_{f \cap g}}\vev{\xi}^{2s-\delta}}_{q'}\norm{\vev{\xi}^\delta \abs{\hat{f}-\hat{g}}^2}_{q}\end{equation}
for any Hölder conjugates $1 =\frac{1}{q}+\frac{1}{q'}$ and any $\delta \in \R$. We fix $\delta$ and $q$ later so that we may estimate $\norm{\vev{\xi}^\delta \abs{\hat{f}-\hat{g}}^2}_{q}$ with the $H^{t,p}$ norm. To do so, we use the identity \begin{equation}\norm{\vev{\xi}^\delta \abs{\hat{f}-\hat{g}}^2}_{q}=\norm{\vev{\xi}^{\delta/2} \abs{\hat{f}-\hat{g}}}_{2q}^2.\label{eq:square-simple-identity}
\end{equation}

Choose $q=\frac{p'}{2}$ where $p'$ is the Hölder conjugate of $p$. Then we may calculate that $q' = \frac{p}{2-p}$. Moreover, we want that $\delta/2 \leq t$, and for our purposes the optimal choice is $\delta = 2t$. Therefore, we obtain from \eqref{eq:continuity-of-Fourier}, \eqref{eq:hölder-inequality} and \eqref{eq:square-simple-identity} with $\delta=2t$ and $q=\frac{p'}{2}$, the estimate
\begin{equation}\norm{M_{f\cap g}(f-g)}_{H^s(\R^n)}^2 \leq c_{n,p}\norm{\chi_{A_{f \cap g}}\vev{\xi}^{2s-2t}}_{\frac{p}{2-p}}\norm{f-g}_{H^{t,p}(\R^n)}^2.\end{equation}

Finally, we note that $s < t-n(\frac{1}{p}-\frac{1}{2})$ if and only if $\norm{\vev{\xi}^{2s-2t}}_{\frac{p}{2-p}} < \infty$. Also, if $A_{f \cap g}$ has finite measure and $s \leq t$, or $A_{f \cap g}$ is compact, then $\norm{\chi_{A_{f \cap g}}\vev{\xi}^{2s-2t}}_{\frac{p}{2-p}}$ is finite for any $p \in [1,2]$. This completes the proof.
\end{proof}

We are now ready to prove our main theorem in this note.

\begin{proof}[Proof of Theorem \ref{thm:HsBesselGroupthm}] Suppose that $p \in [1,2]$ and $t,s \in \R$. Let $P, Q \in G$. If $Pf-Qg \notin H^{t,p}(\R^n)$ and $A_{f \cap g}$ has positive measure, then the estimate \begin{equation}
    d([f],[g])^2 \leq \norm{\vev{\xi}^s(|\hat{f}|-\abs{\hat{g}})}_2^2 + c_{n,p}\norm{\chi_{A_{f \cap g}}\vev{\xi}^{2s-2t}}_{\frac{p}{2-p}}\norm{Pf-Qg}_{H^{t,p}(\R^n)}^2
\end{equation} holds trivially as the right-hand side is infinite (by the convention used in this note). If $A_{f \cap g}$ has measure zero, then $ d([f],[g])^2 \leq \norm{\vev{\xi}^s(|\hat{f}|-\abs{\hat{g}})}_2^2$ by a simple application of Lemma \ref{lem:H^slemma} and the definition of $G$. Hence, we may suppose that $Pf-Qg \in H^{t,p}(\R^n)$ without the loss of generality.

It follows from the definition of $G$ that $\abs{\widehat{Pf}}=\abs{\hat{f}}$, $\abs{\widehat{Qg}}=\abs{\hat{g}}$, and $A_{Pf\cap Qg} = A_{f\cap g}$. Since $Pf$ and $Qg$ satisfy the assumptions of Lemma \ref{lem:HsBesselthm}, we get
\begin{equation}\begin{split}
    \norm{Pf-Qg}_{H^s(\R^n)}^2
    &\leq \norm{\vev{\xi}^s(|\widehat{f}|-\abs{\widehat{g}})}_2^2 \\
    &\quad+ c_{n,p}\norm{\chi_{A_{f \cap g}}\vev{\xi}^{2s-2t}}_{\frac{p}{2-p}}\norm{Pf-Qg}_{H^{t,p}(\R^n)}^2.
    \end{split}
\end{equation}
We may now take $\inf_{P,Q \in G}$ from both sides to obtain the desired estimate.\end{proof}

\appendix
\section{On estimates using only Fourier magnitude difference data}
\label{sec:appendix1}

In this appendix, we consider conditional estimates that only use the Fourier magnitude difference data with a constant depending on one of the functions. The main point is to illustrate that such estimates are possible, as expected, but they seem to hold for somewhat trivial reasons. Let $r \in [0,1)$ and $h \in H^s(\R^n)$. We define the following two sets
\begin{equation}X^s(h;r) := \{\, g\in H^s(\R^n)\,;\, \norm{M_{h \cap g}(h-g)}_{H^s(\R^n)} \leq r\norm{h-g}_{H^s(\R^n)}\, \}.\end{equation}
and
\begin{equation}
\begin{split}
X^s(r) :=\,&\{\, (f,g) \in H^s(\R^n) \times H^s(\R^n)\,;\,\\
&\quad \norm{M_{f\cap g}(f-g)}_{H^s(\R^n)} \leq r\norm{f-g}_{H^s(\R^n)}\,\}.
\end{split}
\end{equation}
The parameter $r$ measures how much the functions $f$ and $g$ are disjoint in the suitable $H^s$ sense. We will also calculate the value of $r$ for each given pair of functions in the end. We omit the proofs in this section as those are rather elementary.

\begin{lemma}\label{lem:Hsstability} Let $n \in \Z_+$, $s \in \R$ and $r \in [0,1)$. Let $f \in H^s(\R^n)$.
If $g \in X^s(f;r)$, then 
\begin{equation}\label{eq:stabestimate}
\norm{f-g}_{H^s(\R^n)}^2 \leq C\norm{\vev{\xi}^s(|\hat{f}|-\abs{\hat{g}})}_2^2.
\end{equation}
where $C = 1/(1-r^2)$. The same estimate holds for the pairs of functions $(f,g) \in X^s(r)$.
\end{lemma}

We next consider when such $r \in [0,1)$ exists and how to estimate $r$. We first remark that if $\supp(\hat{f}) = \supp(\hat{g})$, then the inequality associated with $X^s(f;r)$ or $X^s(r)$ holds only if $r = 1$, and one cannot use this simple argument to obtain the estimate \eqref{eq:stabestimate}. The content of the next lemma is that such an estimate always holds if $\supp(\hat{f}) \neq \supp(\hat{g})$.
\begin{lemma}\label{lem:Hsstability2} Let $n \in \Z_+$ and $s \in \R$. Let $f, g \in H^s(\R^n)$ be such that $\supp(\hat{f}) \neq \supp(\hat{g})$. Then for any $r \in [r_0,1)$ it holds that $(f,g) \in X^s(r)$ where
\begin{equation}
    r_0 = \sqrt{1-\frac{\norm{M_{f\setminus g}f}_{H^s(\R^n)}^2 + \norm{M_{g\setminus f}g}_{H^s(\R^n)}^2}{\norm{f-g}_{H^s(\R^n)}^2}}.
\end{equation}
Moreover, $(f,g) \notin X^s(r)$ when $r \in [0,r_0)$.
\end{lemma}

One could derive similar estimates for the spaces $X^s(f;r)$ and $X^s(g;r)$ but we will omit this.

\begin{remark}\label{rmk:Hsstability} Let $n \in \Z_+$ and $s \in \R$. Let $G \subset \mathcal{F}^s (\R^n)$ be a subgroup. If $f,g \in H^s(\R^n)$ are such that $\supp(\hat{f}) \neq \supp(\hat{g})$, then Lemmas \ref{lem:Hsstability} and \ref{lem:Hsstability2} imply
\begin{equation}\label{eq:stabestimate2}
d([f],[g])^2 \leq \frac{d([f],[g])^2}{\norm{M_{f\setminus g}f}_{H^s(\R^n)}^2 + \norm{M_{g\setminus f}g}_{H^s(\R^n)}^2}\norm{\vev{\xi}^s(|\hat{f}|-\abs{\hat{g}})}_2^2.
\end{equation}
This in fact reveals that the estimate of this type reduces to a trivial property \begin{equation}\frac{\norm{\vev{\xi}^s(|\hat{f}|-\abs{\hat{g}})}_2^2}{\norm{M_{f\setminus g}f}_{H^s(\R^n)}^2 + \norm{M_{g\setminus f}g}_{H^s(\R^n)}^2}\geq 1.\end{equation} One cannot even hope to do any better with this approach since the terms $\norm{M_{f\setminus g}f}_{H^s(\R^n)}^2$ and $\norm{M_{g\setminus f}g}_{H^s(\R^n)}^2$ are invariant under the operations by $G$. On the other hand, when $\supp(\hat{f}) \to \supp(\hat{g})$, one easily notes that the constant is not well-controlled even if $d([f],[g])$ is a priori assumed to be bounded.
\end{remark}

\section{Example of quotient distances and the Fourier phase problem}\label{sec:genQuotientL2}
We illustrate here that the Fourier phase problem might not be mathematically very interesting if the function space is large and the quotient distance is taken with respect to most of the possible sources of non-uniqueness. Let us define that $[f]=[g]$ if and only if $\abs{\hat{f}}=\abs{\hat{g}}$. Let the corresponding $L^2$ distance be defined as $d([f],[g]) =\inf_a \norm{f-M_ag}_2$ where $M_a$ is the Fourier multiplier associated with a measurable function $a: \R^n \to S^1$. Note that automatically any measurable $a$ with the range $S^1$ generates an $L^2$ Fourier multiplier keeping the Fourier magnitude data invariant. 

Write \begin{equation}A = \{\,\xi \in \R^n\,;\, \hat{f}(\xi) = 0 \text{ or } \hat{g}(\xi)=0\,\}.\end{equation} For any fixed $\xi \in \R^n$ with $\hat{f}(\xi)\neq0$, one has
\begin{equation}
    \abs{\hat{f}-\hat{g}}=\abs{\hat{f}}^{-1}\abs{\abs{\hat{f}}^2-\hat{g}\overline{\hat{f}}}
\label{eq:smallestimate}
\end{equation}
This holds especially if $\xi \in \R^n \setminus A$. Let us therefore define the mapping $a: \R^n \to S^1$ as \[a=\chi_{\R^n \setminus A}\hat{f}\abs{\hat{f}}^{-1}\overline{\hat{g}}\abs{\hat{g}}^{-1} + \chi_A.\] 
Using \eqref{eq:smallestimate}, we easily obtain the identity
\[\norm{f-M_ag}_2 = \norm{\abs{\hat{f}}-\abs{\hat{g}}}_2.\] Since $M_a$ is one of the suitable Fourier multipliers for any $f,g \in L^2(\R^n)$, we find eventually that
\[d([f],[g]) = \norm{\abs{\hat{f}}-\abs{\hat{g}}}_2\]
holds for all $f,g\in L^2(\R^n)$. The same principles hold of course on $H^s(\R^n)$.

\end{document}